\documentclass[11pt,reqno]{amsart}

\usepackage[margin=1in]{geometry}

\usepackage{amssymb, mathrsfs, amscd}
\usepackage{extpfeil}

\usepackage{graphicx}
\usepackage{colortbl}
\usepackage{float}
\usepackage{caption}
\usepackage{subcaption} 
\graphicspath{{Figure/}}
\usepackage[dvipsnames]{xcolor}

\usepackage{enumerate}
\usepackage{url}
\usepackage{hyperref}
\usepackage{doi}

\theoremstyle{plain}
\newtheorem{theorem}{Theorem}[section]
\newtheorem{proposition}[theorem]{Proposition}
\newtheorem{lemma}[theorem]{Lemma}
\newtheorem{corollary}[theorem]{Corollary}

\theoremstyle{definition}
\newtheorem{definition}[theorem]{Definition}

\theoremstyle{remark}
\newtheorem{remark}[theorem]{Remark}

\numberwithin{equation}{section}

\makeatletter
\@namedef{subjclassname@2020}{\textup{2020} Mathematics Subject Classification}
\makeatother

\title[viscous-dispersive shocks for the NSFK system]{Existence and uniqueness of viscous-dispersive shock profiles for the Navier--Stokes--Fourier--Korteweg system}


\author[S. Han]{Sungho Han}
\address[Sungho Han]{Department of Mathematical Sciences, \newline
Korea Advanced Institute of Science and Technology (KAIST), \newline
291 Daehak-ro, Yuseong-gu, Daejeon 34141, Republic of Korea}
\email{sungho\_han@kaist.ac.kr}

\begin{document}

\bibliographystyle{acm}

\begin{abstract}
We establish the existence and uniqueness, up to translation, of monotone viscous-dispersive shock profiles for the one-dimensional Navier--Stokes--Fourier--Korteweg system, provided the two end states satisfy the Rankine--Hugoniot and Lax entropy conditions and the shock strength is small. Taking the shock strength as a singular perturbation parameter, we reduce the profile equations by Fenichel's geometric singular perturbation theory to a scalar equation on a slow manifold and obtain the profile as a heteroclinic orbit. We also give pointwise estimates of the profiles and their derivatives up to fourth order, with explicit leading-order terms for the first derivatives.
\end{abstract}

\maketitle

\section{Introduction}\label{sec:intro}
\setcounter{equation}{0}

We consider the one-dimensional Navier--Stokes--Fourier--Korteweg (NSFK) system in Lagrangian mass coordinates $(t,x)\in\mathbb R_+\times\mathbb R$, where $\mathbb R_+:=(0,\infty)$:
\begin{equation}\label{eq:NSFK}
\begin{cases}
    \displaystyle v_t - u_x = 0, \\
    \displaystyle u_t + p(v,\theta)_x = \mu\left(\frac{u_x}{v}\right)_x
        + \kappa\left(-\frac{v_{xx}}{v^5} + \frac{5v_x^2}{2v^6}\right)_x, \\
    \displaystyle E_t + (p(v,\theta)u)_x = \alpha\left(\frac{\theta_x}{v}\right)_x + \mu\left(\frac{uu_x}{v}\right)_x + \kappa\left[\left(-\frac{v_{xx}}{v^5} + \frac{5v_x^2}{2v^6}\right)u + \frac{v_x u_x}{v^5}\right]_x.
\end{cases}
\end{equation}
Here, $v=v(t,x)$, $u=u(t,x)$, and $\theta=\theta(t,x)$ denote the specific volume, velocity, and absolute temperature, respectively, and $E=E(t,x)$ denotes the specific total energy. The constants $\mu,\alpha,\kappa>0$ are the viscosity, heat-conductivity, and capillarity coefficients, respectively.

For an ideal polytropic gas, the pressure is given by
\begin{equation}\label{eq:state_laws}
    p(v,\theta)=\frac{R\theta}{v},
\end{equation}
where $R>0$ is the gas constant. The specific total energy is
\begin{equation}\label{eq:total_energy}
    E = e+\frac{u^2}{2} = \frac{R\theta}{\gamma-1} +\frac{\kappa v_x^2}{2v^5} +\frac{u^2}{2},
\end{equation}
where $\gamma>1$ is the adiabatic exponent and $e$ is the internal energy, which includes the capillary energy $\kappa v_x^2/(2v^5)$.

The NSFK system describes compressible flows with internal capillarity, in which the Korteweg stress depends on the density and its spatial derivatives. The diffuse-interface description of capillary fluids goes back to van der Waals \cite{W94} and Korteweg \cite{K01}. Dunn and Serrin \cite{DS85} introduced an interstitial work flux into the energy
balance, thereby obtaining a formulation consistent with the second law of thermodynamics.

From \eqref{eq:NSFK} and \eqref{eq:total_energy}, we obtain 
\begin{equation*}
\begin{cases}
    \displaystyle v_t-u_x=0,\\
    \displaystyle u_t+p(v,\theta)_x =\mu\left(\frac{u_x}{v}\right)_x+\kappa\left(-\frac{v_{xx}}{v^5}+\frac{5v_x^2}{2v^6}\right)_x,\\
    \displaystyle \frac{R}{\gamma-1}\theta_t+p(v,\theta) u_x=\alpha\left(\frac{\theta_x}{v}\right)_x+\frac{\mu u_x^2}{v}.
\end{cases}
\end{equation*}

\subsection{Literature Review and Motivation}
\label{subsec:motivation}

For the barotropic Navier--Stokes--Korteweg (NSK) system, the existence theory is developed in \cite{DD01,H11,CH13,GL16,AS22}. Nonlinear stability is known for rarefaction waves \cite{C12}, viscous-dispersive shock profiles \cite{CHZ15,HKKL25}, and composite waves of a rarefaction wave and a viscous-dispersive shock profile \cite{HK26}. Spectral stability of viscous-dispersive shock profiles in Korteweg-type and quantum-hydrodynamic models is investigated in \cite{FLP26,FPZ23}. For the full NSFK system, the existence theory is developed in \cite{HL96,H09,CHZ17,HPZ18}, and stability is known for a single viscous contact wave \cite{CX13} and for a composite wave of two rarefaction waves and a viscous contact wave \cite{CS19}.

A viscous-dispersive shock profile may be monotone or have oscillatory tails. For the scalar Korteweg--de Vries--Burgers (KdVB) equation, Bona and Schonbek \cite{BS85} characterized the transition between the two behaviors. The stability of monotone KdVB profiles is studied in \cite{Pego85,CEKS26-L2}, and the structure and stability of oscillatory KdVB profiles in \cite{CEKS26-Osc}. For the modified KdVB equation, whose flux is non-convex, quantitative properties of monotone degenerate shock profiles and the stability of their composite waves with a rarefaction wave are obtained in \cite{EHK26}.

At the system level, Smoller and Shapiro \cite{SS82} obtained early existence results under constant viscosity and dispersion matrices, and Pego \cite{P84} constructed weak shock profiles of viscous systems with singular viscosity matrices, including the Navier--Stokes--Fourier system, by a center manifold reduction. When only one of heat conduction and capillarity is present, the profile equations form a planar system. For the Navier--Stokes--Fourier system, Gilbarg \cite{G51} proved the existence and uniqueness of profiles for arbitrary end states satisfying the shock conditions. Hagan and Slemrod \cite{HS83} obtained profiles of Lax shocks of arbitrary amplitude for the isothermal system with capillarity. Folino, Lattanzio, and Plaza \cite{FLP26} established the existence and uniqueness up to translation of profiles of arbitrary amplitude for the barotropic NSK system and obtained the monotonicity and derivative estimates of weak profiles by a slow manifold reduction.

With both heat conduction and capillarity, the profile equations form a three-dimensional system. For Lax shocks of an ideal gas, the equilibria of the profile equations are classified and connecting orbits are computed numerically in \cite{ThH14}. In this paper, we prove the existence of monotone profiles of weak Lax shocks, their uniqueness in the monotone class, and pointwise estimates. We obtain these results by Fenichel's geometric singular perturbation theory \cite{F79,J95}, which gives weak shock profiles in \cite{HKKL25,EEKO25,EE26}.

The method of $a$-contraction with shifts was introduced by Kang and Vasseur \cite{KV17,KV21-JEMS} and extended to composite waves by Kang, Vasseur, and Wang \cite{KVW23}. It has been applied to viscous-dispersive shocks of the NSK system \cite{HKKL25,HK26,HKO26,EKK26} and of KdVB-type equations \cite{CEKS26-L2,CEKS26-Osc,EHK26}. The companion work \cite{HKK26} uses this method, together with the monotonicity and the pointwise estimates of Theorem~\ref{thm:main}, to prove the time-asymptotic stability of a composite wave of a rarefaction wave, a viscous contact wave, and a viscous-dispersive shock profile for the NSFK system, the analogue of the result of Kang, Vasseur, and Wang \cite{KVW25} for the Navier--Stokes--Fourier system. Since the left end state of the shock there ranges over a neighborhood of a given state, the constants in Theorem~\ref{thm:main} are needed uniformly for left end states in a compact set.

\subsection{Viscous-Dispersive Shock Profiles}
\label{subsec:traveling_waves}

We consider viscous-dispersive shock profiles of the NSFK system, defined as traveling wave solutions $(\tilde v,\tilde u,\tilde\theta)(\xi)$ with $\xi=x-\sigma t$ connecting two distinct constant end states $(v_\pm,u_\pm,\theta_\pm)$ as $\xi\to\pm\infty$. 

Since the end states are constant, their capillary contribution to the total energy vanishes. We therefore define
\begin{equation*}
    p_\pm := p(v_\pm,\theta_\pm) = \frac{R\theta_\pm}{v_\pm}, \qquad E_\pm := \frac{R}{\gamma-1}\theta_\pm +\frac{u_\pm^2}{2}.
\end{equation*}
The end states and the shock speed satisfy the Rankine--Hugoniot condition:
\begin{equation}\label{eq:RH}
\begin{cases}
    -\sigma(v_+-v_-)-(u_+-u_-)=0,\\
    -\sigma(u_+-u_-)+(p_+-p_-)=0,\\
    -\sigma(E_+-E_-)+(p_+u_+-p_-u_-)=0.
\end{cases}
\end{equation}
The characteristic speeds of the underlying Euler system in Lagrangian mass coordinates are
\begin{equation*}
    \lambda_1=-\sqrt{\frac{\gamma p(v,\theta)}{v}},\qquad
    \lambda_2=0,\qquad
    \lambda_3=\sqrt{\frac{\gamma p(v,\theta)}{v}}.
\end{equation*}
Thus, the Lax shock conditions for the 1- and 3-shock families are
\begin{equation}\label{eq:Lax}
    -\sqrt{\frac{\gamma p_+}{v_+}}<\sigma<-\sqrt{\frac{\gamma p_-}{v_-}}
    \quad\text{for a 1-shock},
    \qquad
    \sqrt{\frac{\gamma p_+}{v_+}}<\sigma<\sqrt{\frac{\gamma p_-}{v_-}}
    \quad\text{for a 3-shock}.
\end{equation}
It follows from the Rankine--Hugoniot condition \eqref{eq:RH} and the Lax shock conditions \eqref{eq:Lax} that
\begin{equation*}
\begin{aligned}
    &\text{1-shock:}\qquad \sigma<0,\quad v_+<v_-,\quad u_+<u_-,\quad \theta_+>\theta_-,\\
    &\text{3-shock:}\qquad \sigma>0,\quad v_+>v_-,\quad u_+<u_-,\quad \theta_+<\theta_-.
\end{aligned}
\end{equation*}

For traveling waves, the NSFK system \eqref{eq:NSFK} becomes the system of ordinary differential equations
\begin{equation}\label{eq:shock-ode}
\begin{cases}
    \displaystyle -\sigma\tilde v'-\tilde u'=0,\\
    \displaystyle -\sigma\tilde u'+\tilde p' =
    \left(\mu\frac{\tilde u'}{\tilde v}\right)' +
    \kappa\left( -\frac{\tilde v''}{\tilde v^5} +\frac{5(\tilde v')^2}{2\tilde v^6} \right)',\\
    \displaystyle
    -\sigma\tilde E' +(\tilde p\tilde u)'
    = \left(\alpha\frac{\tilde\theta'}{\tilde v}\right)' + \left(\mu\frac{\tilde u\tilde u'}{\tilde v}\right)' + \kappa\left[ \left(-\frac{\tilde v''}{\tilde v^5} +\frac{5(\tilde v')^2}{2\tilde v^6} \right)\tilde u +\frac{\tilde v'\tilde u'}{\tilde v^5} \right]',
\end{cases}
\end{equation}
where $\tilde{p} := p(\tilde{v}, \tilde{\theta})$ and $\tilde{E} := \frac{R}{\gamma-1}\tilde{\theta}  + \frac{\kappa}{2}\frac{(\tilde{v}')^2}{\tilde{v}^5} + \frac{\tilde{u}^2}{2}$. The profile is subject to the far-field conditions
\begin{equation}\label{eq:shock-ode-bc}
    \lim_{\xi\to\pm\infty} (\tilde v,\tilde u,\tilde\theta)(\xi)= (v_\pm,u_\pm,\theta_\pm).
\end{equation}
A solution of \eqref{eq:shock-ode}--\eqref{eq:shock-ode-bc} is called a monotone profile if each of $\tilde v$, $\tilde u$, and $\tilde\theta$ is monotone on $\mathbb R$.

\subsection{Main Results} \label{subsec:main_results}
The following theorem establishes the existence and uniqueness up to translation of small-amplitude monotone viscous-dispersive shock profiles, together with pointwise estimates for the normalized profile.

\begin{theorem} \label{thm:main}
    For a given left end state $(v_-,u_-,\theta_-)\in\mathbb R_+\times\mathbb R\times\mathbb R_+$, there exist positive constants $\varepsilon_0$, $c$, and $C$ such that the following holds for every right end state $(v_+,u_+,\theta_+)\in\mathbb R_+\times\mathbb R\times\mathbb R_+$ and every shock speed $\sigma\in\mathbb R$ satisfying the Rankine--Hugoniot condition \eqref{eq:RH}, the Lax 3-shock condition in \eqref{eq:Lax}, and $\varepsilon:=|v_+-v_-|<\varepsilon_0$.
    There is a traveling wave solution $(\tilde{v}, \tilde{u}, \tilde{\theta}) : \mathbb{R} \to \mathbb R_+ \times \mathbb{R} \times \mathbb R_+$ to \eqref{eq:shock-ode}--\eqref{eq:shock-ode-bc}, which is a monotone profile and is unique up to translation within the monotone class. Let $(\tilde{v}, \tilde{u}, \tilde{\theta})$ be normalized by $\tilde{v}(0) = \frac{v_- + v_+}{2}$.
    Then
    \begin{equation}\label{eq:shock_estimates}
    \begin{aligned}
        &\tilde{v}_\xi > 0, \quad \tilde{u}_\xi < 0, \quad \tilde{\theta}_\xi < 0, \quad &&\xi \in \mathbb{R},\\
        &|(\tilde{v}(\xi) - v_-, \tilde{u}(\xi) - u_-, \tilde{\theta}(\xi) - \theta_-)|
            \le C\varepsilon\, e^{-c\varepsilon|\xi|}, \quad &&\xi \le 0,\\
        &|(\tilde{v}(\xi) - v_+, \tilde{u}(\xi) - u_+, \tilde{\theta}(\xi) - \theta_+)|
            \le C\varepsilon\, e^{-c\varepsilon|\xi|}, \quad &&\xi \ge 0,\\
        &|(\tilde{v}_\xi, \tilde{u}_\xi, \tilde{\theta}_\xi)|
            \le C\varepsilon^2 e^{-c\varepsilon|\xi|}, \quad &&\xi \in \mathbb{R},\\
        &|\partial_\xi^k(\tilde{v},\tilde{u},\tilde{\theta})|
            \le C\varepsilon^{k-1}\,|\tilde{v}_\xi|,
            \quad &&k=2,3,4,\quad \xi \in \mathbb{R}.
    \end{aligned}
    \end{equation}
    Set $\sigma_-:=\sqrt{\gamma p_-/v_-}$ and $C_*:=\frac{R\gamma(\gamma+1)\sigma_-}{2(\mu R\gamma+\alpha(\gamma-1)^2)}$. Then the shock speed satisfies
    \begin{equation}\label{eq:sigma-estimate}
        |\sigma-\sigma_-|\le C\varepsilon,
    \end{equation}
    the derivatives satisfy
    \begin{equation}\label{eq:derivative-comparability}
        C^{-1}|\tilde v_\xi| \le |\tilde u_\xi| \le C|\tilde v_\xi|, \qquad
        C^{-1}|\tilde v_\xi| \le |\tilde\theta_\xi| \le C|\tilde v_\xi|, \qquad \xi\in\mathbb{R},
    \end{equation}
    and, more precisely,
    \begin{equation}\label{eq:derivative_relations}
        |\tilde{u}_\xi + \sigma_- \tilde{v}_\xi| \le C\varepsilon|\tilde{v}_\xi|, \qquad \left|\tilde{\theta}_\xi + \frac{(\gamma-1)p_-}{R}\tilde{v}_\xi\right| \le C\varepsilon|\tilde{v}_\xi|, \qquad \xi \in \mathbb{R}.
    \end{equation}
    Moreover,
    \begin{equation}\label{eq:logistic-signed}
        C^{-1}(v_+ - \tilde{v})(\tilde{v} - v_-) \le \tilde{v}_\xi \le C(v_+ - \tilde{v})(\tilde{v} - v_-), \qquad \xi \in \mathbb{R},
    \end{equation}
    and, more precisely,
    \begin{equation}\label{eq:logistic-precise}
        \bigl|\tilde v_\xi-C_*(v_+-\tilde v)(\tilde v-v_-)\bigr|\le C\varepsilon\,(v_+-\tilde v)(\tilde v-v_-),\qquad \xi\in\mathbb R.
    \end{equation}
    Finally, if $\Omega\subset\{(v,\theta):v>0,\ \theta>0\}$ is compact, then $\varepsilon_0$, $c$, and $C$ can be chosen uniformly for all left end states with $(v_-,\theta_-)\in\Omega$.
\end{theorem}

\begin{remark}
\label{rem:main}

The change of variables $(\xi,u,\sigma)\mapsto(-\xi,-u,-\sigma)$ maps a 1-shock profile to a 3-shock profile, with the end states changing as
\[
    \bigl((v_-,u_-,\theta_-),\,(v_+,u_+,\theta_+)\bigr)\ \mapsto\ \bigl((v_+,-u_+,\theta_+),\,(v_-,-u_-,\theta_-)\bigr).
\]
Hence the conclusions of Theorem~\ref{thm:main} hold for 1-shocks, with $\tilde v_\xi<0$, $\tilde u_\xi<0$, $\tilde\theta_\xi>0$ in place of the signs in \eqref{eq:shock_estimates}, with $-\tilde v_\xi$ in place of $\tilde v_\xi$ in \eqref{eq:logistic-signed}, and with $\sigma_-$ and $p_-$ replaced by $-\sqrt{\gamma p_+/v_+}$ and $p_+$ elsewhere, including in $C_*$. Since the reflected left end state depends on $\varepsilon$, we apply the last statement of Theorem~\ref{thm:main} with a compact set $\Omega$ that contains $(v_+,\theta_+)$ for $\varepsilon\le\frac{v_-}{\gamma+1}$. In the rest of the paper, we consider 3-shocks and assume
\[
 v_-<v_+,\qquad \varepsilon=v_+-v_-,\qquad \sigma=
\sqrt{-\frac{p_+-p_-}{v_+-v_-}}>0.
\]
\end{remark}

The following corollary provides the second-order estimate used in the companion stability analysis \cite{HKK26}.

\begin{corollary}
\label{cor:second-order}
For a given left end state $(v_-,u_-,\theta_-)\in\mathbb R_+\times\mathbb R\times\mathbb R_+$, there exist positive constants $\varepsilon_0$ and $C$ such that the following holds for every right end state $(v_+,u_+,\theta_+)\in\mathbb R_+\times\mathbb R\times\mathbb R_+$ and every shock speed $\sigma\in\mathbb R$ satisfying the Rankine--Hugoniot condition \eqref{eq:RH}, the Lax 3-shock condition in \eqref{eq:Lax}, and $\varepsilon:=|v_+-v_-|<\varepsilon_0$.
Let $(\tilde v,\tilde u,\tilde\theta)$ be the traveling wave solution of Theorem~\ref{thm:main} and define $A_*:=\frac{\gamma(\gamma+1)p_-}{2v_-^2\sigma_-}$. Then
\begin{equation}
\label{eq:secant-estimate}
\left|
    \frac{\tilde p-p_+}{\tilde v-v_+}
    -
    \frac{\tilde p-p_-}{\tilde v-v_-}
    -
    \sigma_-A_*
    \frac{\mu R\gamma}
         {\mu R\gamma+\alpha(\gamma-1)^2}
    (v_+-v_-)
\right|
\le C\varepsilon^2,
\qquad \xi\in\mathbb R.
\end{equation}
Moreover, if $\Omega\subset\{(v,\theta):v>0,\ \theta>0\}$ is compact, then $\varepsilon_0$ and $C$ can be chosen uniformly for all left end states with $(v_-,\theta_-)\in\Omega$.
\end{corollary}

\subsection{Main Ideas of the Proof}
We integrate the profile equations \eqref{eq:shock-ode} once. The integrated mass and momentum equations express $\tilde u$ and $\tilde\theta$ in terms of $\tilde v$, and the integrated energy equation becomes a third-order equation for $\tilde v$. As in \cite{EEKO25}, we set $\tilde v(\xi)=v_-+\varepsilon\varphi_0(z)$, so that $z=\varepsilon\xi$ is slow while $\xi$ is fast, and take $\varphi_0$, $\varphi_1=\dot\varphi_0$, and $\varphi_2=\varepsilon\ddot\varphi_0$ as unknowns. This gives the slow-fast system \eqref{system-slow}. At $\varepsilon=0$, the slow-fast system reduces to algebraic relations, which define the critical manifold
\[
    \bigl\{\varphi_1=C_*\varphi_0(1-\varphi_0),\ \varphi_2=0\bigr\}.
\]

To obtain a locally invariant manifold near the critical manifold for small $\varepsilon>0$, we apply Fenichel's first theorem (Proposition~\ref{prop:fenichel}), which requires the right-hand side of \eqref{system-slow} to be $C^\infty$ up to $\varepsilon=0$. The main difficulty is a term of order $1/\varepsilon$ in this right-hand side. By Lemma~\ref{lem:RH-parametrization}, the Rankine--Hugoniot condition gives $\sigma=\sigma(\varepsilon)$, and the identity \eqref{eq:algebraic-cancellation} for $\sigma(\varepsilon)$ cancels the factor $1/\varepsilon$. To obtain constants uniform in $\Omega$, we adjoin the equation $\dot q=0$ for the left end state $q=(v_-,\theta_-)$ to \eqref{system-slow} and apply the theorem to the resulting enlarged system \eqref{system-enlarged}. On the manifold $M_\varepsilon$ obtained from the theorem, called the slow manifold, the first equation of \eqref{system-slow} becomes the scalar equation
\[
    \frac{d\varphi_0}{dz}=C_*\varphi_0(1-\varphi_0)+O(\varepsilon).
\]
Lemma~\ref{lem:invariance} shows that each zero of the right-hand side gives an equilibrium of \eqref{system-slow}, and Lemma~\ref{lem:endpoint} shows that the right-hand side vanishes exactly at $0$ and $1$ and is positive in between. The profile is the solution of the scalar equation connecting $0$ to $1$.

Uniqueness is proved in the class of all monotone solutions of \eqref{eq:shock-ode}--\eqref{eq:shock-ode-bc}. Such a solution satisfies the integrated profile equations with three integration constants, and the far-field conditions \eqref{eq:shock-ode-bc} prescribe only the limits of the solution. We show that these constants vanish. For the constant in the integrated momentum equation, we use the auxiliary variable of \cite{BDD06}, which satisfies a linear first-order equation with a coefficient bounded below by a positive constant. The solution then lies on the one-dimensional unstable manifold of the left end state, and, as in \cite{EEKO25,EE26}, it is a translate of the constructed profile.

The estimates in Theorem~\ref{thm:main} compare $\tilde v_\xi$ with $(\tilde v-v_-)(v_+-\tilde v)$ up to the end states, where both vanish. Since the equilibria $(0,0,0)$ and $(1,0,0)$ lie on $M_\varepsilon$, the $O(\varepsilon)$ term of the scalar equation vanishes at $\varphi_0=0$ and $\varphi_0=1$ and is therefore of order $\varepsilon\varphi_0(1-\varphi_0)$. Hence
\[
    \tilde v_\xi=\bigl[C_*+O(\varepsilon)\bigr](\tilde v-v_-)(v_+-\tilde v),
\]
where the $O(\varepsilon)$ term is uniform in $\xi\in\mathbb R$. This factorization yields the bounds \eqref{eq:logistic-signed} and \eqref{eq:logistic-precise}, the exponential decay, and Corollary~\ref{cor:second-order}. The bounds on higher derivatives follow from the scalar equation, and the formula for $\tilde\theta$ gives $\tilde\theta_\xi<0$.

\subsection{Organization of the Paper}
The paper is organized as follows. In Section~\ref{sec:existence}, we recall the geometric singular perturbation framework, reduce the profile equations to a slow-fast system, verify its hypotheses, and construct the monotone profile from the scalar equation on the slow manifold. In Section~\ref{sec:uniqueness_estimates}, we prove the uniqueness of monotone profiles, the pointwise estimates in Theorem~\ref{thm:main}, and Corollary~\ref{cor:second-order}.

\subsection{Notation}
\begin{itemize}
\item A dot denotes differentiation with respect to the slow variable $z=\varepsilon\xi$ of Section~\ref{sec:existence}.
\item In Sections~\ref{sec:existence} and~\ref{sec:uniqueness_estimates}, $C$ and $c$ denote positive constants that may change from line to line, depend only on $\Omega$ and $\gamma,R,\mu,\alpha,\kappa$, and are independent of $\varepsilon$ and $\xi$.
\item For nonnegative quantities $A$ and $B$, we write $A\sim B$ if $C^{-1}B\le A\le CB$ with such a constant $C$.
\item We write $A=O(B)$ if $|A|\le C|B|$ with such a constant $C$.
\end{itemize}

\section{Existence of Monotone Traveling Waves} \label{sec:existence}
\setcounter{equation}{0}

We fix a compact set $\Omega\subset\{(v,\theta):v>0,\ \theta>0\}$ and prove the last statements of Theorem~\ref{thm:main} and Corollary~\ref{cor:second-order} for this $\Omega$ in Sections~\ref{sec:existence} and~\ref{sec:uniqueness_estimates}. Taking $\Omega=\{(v_-,\theta_-)\}$ gives the first statements for the given left end state.

\subsection{Geometric Singular Perturbation Framework}
\label{subsec:general-framework}

Let $\mathbf{x}=\mathbf{x}(\xi)\in\mathbb R^n$ and $\mathbf{y}=\mathbf{y}(\xi)\in\mathbb R^\ell$ denote the fast and slow phase variables, respectively, where $\xi$ is the fast coordinate. We consider the fast system
\begin{equation}\label{eq:general_fast_system}
    \begin{cases}
        \displaystyle \frac{d\mathbf{x}}{d\xi} = \mathcal F(\mathbf{x},\mathbf{y},\varepsilon),\\
        \displaystyle \frac{d\mathbf{y}}{d\xi} = \varepsilon\mathcal G(\mathbf{x},\mathbf{y},\varepsilon),
    \end{cases}
\end{equation}
where $\varepsilon$ is a small parameter. For $\varepsilon>0$, reparametrizing the same trajectories by the slow coordinate $z=\varepsilon\xi$, and retaining the notation $\mathbf{x}$ and $\mathbf{y}$, we obtain the equivalent slow system
\begin{equation}\label{eq:general_slow_fast}
    \begin{cases}
        \displaystyle \varepsilon\frac{d\mathbf{x}}{dz} = \mathcal F(\mathbf{x},\mathbf{y},\varepsilon),\\
        \displaystyle \frac{d\mathbf{y}}{dz} = \mathcal G(\mathbf{x},\mathbf{y},\varepsilon).
    \end{cases}
\end{equation}

At $\varepsilon=0$, the limiting fast system is
\begin{equation}\label{eq:general_fast_limit}
    \frac{d\mathbf{x}}{d\xi} = \mathcal F(\mathbf{x},\mathbf{y},0), \qquad \frac{d\mathbf{y}}{d\xi} = 0.
\end{equation}
We choose a subset $M_0$ of its equilibrium set $\{\mathcal F(\mathbf{x},\mathbf{y},0)=0\}$ and impose the following hypotheses:
\begin{itemize}
\item[\textnormal{(H1)}]
The functions $\mathcal F$ and $\mathcal G$ are $C^\infty$ on $U\times I$, where $U\subset\mathbb R^{n+\ell}$ is open and $I\subset\mathbb R$ is an open interval containing $0$.

\item[\textnormal{(H2)}]
The set $M_0$ is a compact $C^\infty$ manifold, possibly with boundary, and is normally hyperbolic relative to the limiting fast system \eqref{eq:general_fast_limit}.

\item[\textnormal{(H3)}]
The manifold $M_0$ is the graph of a $C^\infty$ function $h^0:Y\to\mathbb R^n$:
\[
    M_0 = \left\{ (\mathbf{x},\mathbf{y}): \mathbf{x}=h^0(\mathbf{y}),\ \mathbf{y}\in Y \right\},
\]
where $Y\subset\mathbb R^\ell$ is a compact, simply connected domain with smooth boundary.
\end{itemize}

Under \textnormal{(H3)}, normal hyperbolicity means that $D_{\mathbf{x}}\mathcal F\bigl(h^0(\mathbf{y}),\mathbf{y},0\bigr)$ has no eigenvalue on the imaginary axis for $\mathbf{y}\in Y$. Since $Y$ is compact and this matrix depends continuously on $\mathbf{y}$, the separation from the imaginary axis is uniform on $Y$.

\begin{definition}[\cite{J95}; see also {\cite[Definition~2.2]{EEKO25}}]\label{def:loc-inv}
Consider an autonomous system of ordinary differential equations with a locally Lipschitz vector field. A set $M$ is locally invariant under the system if there is an open set $V\supset M$ with the following property. If a trajectory of the system starts on $M$ and stays in $V$ on a time interval $[0,T]$, then it stays on $M$ on $[0,T]$. The same holds with $[0,T]$ replaced by $[-T,0]$.
\end{definition}

The following proposition states Fenichel's first theorem. See \cite{F79} and \cite[Theorems~1 and~2]{J95}. See also \cite[Proposition~2.1 and Remark~2.1]{EEKO25} and \cite[Proposition~2.1 and Remark~2.1]{EE26}.

\begin{proposition}\label{prop:fenichel}
Fix a finite integer $r\ge1$ and assume \textnormal{(H1)--(H3)}. There exist $\varepsilon_1>0$ and a family $h^\varepsilon:Y\to\mathbb R^n$ for $|\varepsilon|<\varepsilon_1$, such that $(\mathbf{y},\varepsilon)\mapsto h^\varepsilon(\mathbf{y})$ is jointly $C^r$ and $h^0$ is the graph function appearing in \textnormal{(H3)}. For every $0<|\varepsilon|<\varepsilon_1$, the graph
\[
    \left\{ (\mathbf{x},\mathbf{y}): \mathbf{x}=h^\varepsilon(\mathbf{y}),\ \mathbf{y}\in Y \right\}
\]
is contained in an $\ell$-dimensional manifold $M_\varepsilon$ that is locally invariant under the fast system \eqref{eq:general_fast_system}.
\end{proposition}

For $\varepsilon>0$, the fast and slow systems differ only by the rescaling $z=\varepsilon\xi$, so $M_\varepsilon$ is also locally invariant under \eqref{eq:general_slow_fast}.

We apply Proposition~\ref{prop:fenichel} with $r=6$. Setting $s(\mathbf{y},\varepsilon):=\int_0^1\partial_\beta h^\beta(\mathbf{y})|_{\beta=\tau\varepsilon}\,d\tau$, which is jointly $C^5$, we obtain
\begin{equation}\label{eq:fenichel-expansion}
    h^\varepsilon(\mathbf{y}) = h^0(\mathbf{y}) + \varepsilon s(\mathbf{y},\varepsilon).
\end{equation}
In what follows, a function is called smooth if it is $C^5$. This is the regularity used in Sections~\ref{subsubsec:reconstruction} and~\ref{subsubsec:higher}.

\subsection{Profile Equations and Slow-Fast Reduction} \label{subsec:slow-fast-reduction}

We first parametrize the pressure jump $p_+-p_-$ and the shock speed $\sigma$ by the shock strength $\varepsilon$. This parametrization will later be used to verify the regularity hypothesis \textnormal{(H1)} at $\varepsilon=0$.

\begin{lemma}\label{lem:RH-parametrization}
Let $(v_-,u_-,\theta_-)\in\mathbb R_+\times\mathbb R\times\mathbb R_+$ be a left end state, and define
\begin{equation}\label{eq:sigma-explicit}
    \sigma(\varepsilon):=\sqrt{\frac{\gamma p_-}{v_-+\tfrac{\gamma+1}{2}\varepsilon}},
    \qquad |\varepsilon|<\frac{v_-}{\gamma+1}.
\end{equation}
Then the following holds.
\begin{enumerate}[(i)]
\item Let a right end state $(v_+,u_+,\theta_+)\in\mathbb R_+\times\mathbb R\times\mathbb R_+$ and a shock speed $\sigma\in\mathbb R$ satisfy the Rankine--Hugoniot condition \eqref{eq:RH} and the Lax 3-shock condition in \eqref{eq:Lax}, and let $\varepsilon:=v_+-v_-<\frac{v_-}{\gamma+1}$. Then $\varepsilon>0$ and
\begin{equation}\label{eq:p-jump}
    p_+-p_-=-\frac{\gamma p_-\varepsilon}{v_-+\tfrac{\gamma+1}{2}\varepsilon},
    \qquad \sigma=\sigma(\varepsilon).
\end{equation}
\item The function $\sigma$ satisfies
\begin{equation}\label{eq:algebraic-cancellation}
    \sigma(\varepsilon)^2v_--\gamma p_-
    =-\frac{\gamma p_-(\gamma+1)}{2\bigl(v_-+\tfrac{\gamma+1}{2}\varepsilon\bigr)}\varepsilon,
\end{equation}
and it is $C^\infty$ and positive with $\sigma(0)=\sigma_-$. If $(v_-,\theta_-)\in\Omega$, then
\begin{equation}\label{eq:sigma-small-amplitude}
    |\sigma(\varepsilon)-\sigma_-|\le C|\varepsilon|.
\end{equation}
\end{enumerate}
\end{lemma}

\begin{proof}
By the mass and momentum conditions $\eqref{eq:RH}_1$ and $\eqref{eq:RH}_2$, we obtain
\[
    u_+-u_-=-\sigma\varepsilon,\qquad p_+-p_-=-\sigma^2\varepsilon.
\]
Since $\sigma>0$ for a Lax 3-shock, these relations and $R\theta_\pm=p_\pm v_\pm$ from \eqref{eq:state_laws} reduce the energy condition $\eqref{eq:RH}_3$ to
\[
    \frac{p_+v_+-p_-v_-}{\gamma-1}+\frac12(p_++p_-)\varepsilon=0,
    \quad\text{that is,}\quad
    \gamma p_-\varepsilon+\Bigl(v_-+\tfrac{\gamma+1}{2}\varepsilon\Bigr)(p_+-p_-)=0.
\]
If $\varepsilon\le0$, then $p_+\ge p_-$ and $v_+\le v_-$, which contradicts the Lax 3-shock condition in \eqref{eq:Lax}. Hence $\varepsilon>0$, and the identity above is equivalent to the formula for $p_+-p_-$ in \eqref{eq:p-jump}. With $\sigma>0$, the relation $p_+-p_-=-\sigma^2\varepsilon$ then gives $\sigma=\sigma(\varepsilon)$.

The identity \eqref{eq:algebraic-cancellation} and $\sigma(0)=\sigma_-$ follow from \eqref{eq:sigma-explicit}. For $|\varepsilon|<\frac{v_-}{\gamma+1}$, we have
\[
    v_-+\tfrac{\gamma+1}{2}\varepsilon>\tfrac{v_-}{2},
\]
so $\varepsilon\mapsto\sigma(\varepsilon)$ is $C^\infty$ and positive, and $|\sigma'(\varepsilon)|\le C$ for $(v_-,\theta_-)\in\Omega$, which yields \eqref{eq:sigma-small-amplitude}. See also \cite{EEKO25} for the same computation.
\end{proof}

We now reduce the profile equations to a single equation for $\tilde v$. Every solution of the integrated system below solves the profile equations \eqref{eq:shock-ode}, and Section~\ref{subsec:uniqueness} shows that every monotone solution of \eqref{eq:shock-ode}--\eqref{eq:shock-ode-bc} solves the integrated system. The system reads
\begin{subequations}\label{eq:int-system}
\begin{align}
    &-\sigma(\tilde v-v_-)-(\tilde u-u_-)=0, \label{eq:int1}\\
    &-\sigma(\tilde u-u_-)+(\tilde p-p_-)
     =\mu\frac{\tilde u'}{\tilde v} +\kappa\left(-\frac{\tilde v''}{\tilde v^5} +\frac{5(\tilde v')^2}{2\tilde v^6}\right), \label{eq:int2}\\
    &-\sigma(\tilde E-E_-)+(\tilde p\tilde u-p_-u_-)
     =\alpha\frac{\tilde\theta'}{\tilde v} +\mu\frac{\tilde u\tilde u'}{\tilde v} +\kappa\tilde u\left(-\frac{\tilde v''}{\tilde v^5} +\frac{5(\tilde v')^2}{2\tilde v^6}\right) +\kappa\frac{\tilde v'\tilde u'}{\tilde v^5}. \label{eq:int3}
\end{align}
\end{subequations}
The Rankine--Hugoniot condition \eqref{eq:RH} ensures that the right end state is also an equilibrium of \eqref{eq:int-system}.

The integrated mass equation \eqref{eq:int1} gives
\begin{equation} \label{eq:mass-relation}
    \tilde u-u_-=-\sigma(\tilde v-v_-), \qquad \tilde u'=-\sigma\tilde v'.
\end{equation}
With these relations, the integrated momentum and energy equations \eqref{eq:int2} and \eqref{eq:int3} yield
\begin{equation}\label{eq:int3-1}
    -\sigma\left( \frac{R}{\gamma-1}(\tilde\theta-\theta_-) -\frac{\kappa(\tilde v')^2}{2\tilde v^5} -\frac{\sigma^2}{2}(\tilde v-v_-)^2 \right) -\sigma p_-(\tilde v-v_-) =\alpha\frac{\tilde\theta'}{\tilde v}.
\end{equation}
The integrated momentum equation \eqref{eq:int2}, with \eqref{eq:mass-relation} and $\tilde p=R\tilde\theta/\tilde v$, determines the temperature,
\begin{equation}\label{eq:theta_expression}
    \tilde\theta = \frac{\tilde v}{R}\left[ p_- -\sigma^2(\tilde v-v_-) -\mu\sigma\frac{\tilde v'}{\tilde v} +\kappa\left( -\frac{\tilde v''}{\tilde v^5} +\frac{5(\tilde v')^2}{2\tilde v^6} \right) \right].
\end{equation}
Hence, by \eqref{eq:int3-1} and the formula for $\tilde\theta$, $\tilde v$ satisfies the third-order equation
\begin{equation}\label{eq:ODE-v}
\begin{aligned}
    &\frac{\alpha}{R\tilde v}\left( (p_-+\sigma^2v_-)\tilde v' -\frac{25\kappa(\tilde v')^3}{2\tilde v^6} -2\sigma^2\tilde v\tilde v' +9\kappa\frac{\tilde v'\tilde v''}{\tilde v^5} -\mu\sigma\tilde v'' -\frac{\kappa\tilde v'''}{\tilde v^4} \right)\\
    &\qquad =\frac{\kappa\sigma}{2}\frac{(\tilde v')^2}{\tilde v^5} +\frac{\sigma\tilde v}{\gamma-1}\left( \mu\sigma\frac{\tilde v'}{\tilde v} -\frac{5\kappa(\tilde v')^2}{2\tilde v^6} +\sigma^2(\tilde v-v_-) +\kappa\frac{\tilde v''}{\tilde v^5} \right)\\
    &\qquad\quad
      -\frac{\sigma p_-(\tilde v-v_-)}{\gamma-1} +\frac{\sigma^3}{2}(\tilde v-v_-)^2 -\sigma p_-(\tilde v-v_-).
\end{aligned}
\end{equation}

We now introduce the small-amplitude scaling
\begin{equation*}
    \tilde v(\xi)=v_-+\varepsilon\varphi_0(z), \qquad z=\varepsilon\xi,
    \qquad
    \frac{d^k\tilde v}{d\xi^k} = \varepsilon^{k+1}\frac{d^k\varphi_0}{dz^k}
    \quad (k\ge 1).
\end{equation*}
Setting $\varphi_1=\dot\varphi_0$ and $\varphi_2=\varepsilon\ddot\varphi_0=\varepsilon\dot\varphi_1$, equation \eqref{eq:ODE-v} is equivalent to the slow-fast system
\begin{equation}\label{system-slow}
\begin{cases}
    \dot\varphi_0=\varphi_1,\\
    \varepsilon\dot\varphi_1=\varphi_2,\\
    \varepsilon\dot\varphi_2 = f(\varphi_0,\varphi_1,\varphi_2,\varepsilon),
\end{cases}
\end{equation}
where
\begin{equation}\label{eq:f-def}
\begin{aligned}
    &f(\varphi_0,\varphi_1,\varphi_2,\varepsilon;q) \\
    &\quad= -\frac{1}{\varepsilon}g(\varphi_0,\varepsilon;q)
     -\frac{R\sigma^2(\varepsilon\varphi_0+v_-)^5}{\kappa\alpha(\gamma-1)} \left[\frac{\sigma(\gamma+1)}{2}\varphi_0^2 +\mu\varphi_1\right]
     -\frac{R\sigma(\varepsilon\varphi_0+v_-)}{\alpha(\gamma-1)}\varphi_2\\
    &\qquad +\frac{(\varepsilon\varphi_0+v_-)^4}{\kappa} \left[ (p_--\sigma^2v_-)\varphi_1 -2\sigma^2\varepsilon\varphi_0\varphi_1 -\mu\sigma\varphi_2 \right]\\
    &\qquad -\frac{R(\gamma-6)}{2\alpha(\gamma-1)}\sigma\varepsilon^2\varphi_1^2
     +\frac{\varepsilon^2\varphi_1}{\varepsilon\varphi_0+v_-} \left[ -\frac{25}{2}\frac{\varepsilon^2\varphi_1^2}{\varepsilon\varphi_0+v_-} +9\varphi_2 \right],
\end{aligned}
\end{equation}
and
\begin{equation}\label{eq:g-def}
    g(\varphi_0,\varepsilon;q) := \frac{R(\varepsilon\varphi_0+v_-)^5\varphi_0}{\kappa\alpha(\gamma-1)}\sigma\bigl(\sigma^2v_- -\gamma p_-\bigr).
\end{equation}
Here $q:=(v_-,\theta_-)$, and $f$ and $g$ depend on $q$ through $v_-$, $p_-=R\theta_-/v_-$, and the function $\sigma(\varepsilon)$ in \eqref{eq:sigma-explicit}.

\subsection{Hypotheses and Equilibria}\label{subsec:hypotheses}

We adjoin $\dot q=0$ to the slow-fast system \eqref{system-slow}, with $q=(v_-,\theta_-)$ as in \eqref{eq:f-def}, and obtain the enlarged system
\begin{equation}\label{system-enlarged}
\begin{cases}
    \dot\varphi_0=\varphi_1,\qquad \dot q=0,\\
    \varepsilon\dot\varphi_1=\varphi_2,\\
    \varepsilon\dot\varphi_2=f(\varphi_0,\varphi_1,\varphi_2,\varepsilon;q),
\end{cases}
\end{equation}
which has the form of the slow system \eqref{eq:general_slow_fast} with $\mathbf x=(\varphi_1,\varphi_2)$, $\mathbf y=(\varphi_0,q)$, $\mathcal F=(\varphi_2,f)$, and $\mathcal G=(\varphi_1,0)$. Set $K:=[-1,2]$, and let $\hat K$ be a compact convex set with smooth boundary such that
\[
    K\times\Omega\subset\hat K\subset\mathbb R\times\mathbb R_+\times\mathbb R_+ .
\]
The slow variable $(\varphi_0,q)$ ranges over $\hat K$. Thus $\hat K$ is the domain $Y$ in \textnormal{(H3)}, and for each $q\in\Omega$ the slice $\{\varphi_0:(\varphi_0,q)\in\hat K\}$ contains $K$. We verify \textnormal{(H1)--(H3)} for \eqref{system-enlarged}. In what follows, we suppress the dependence on $q$ in the notation.

\subsubsection{Verification of \textnormal{(H1)} and Equilibria for $\varepsilon > 0$}\label{subsubsec:H1}

In \eqref{eq:f-def}, only the term $-g/\varepsilon$ is singular at $\varepsilon=0$. By \eqref{eq:algebraic-cancellation} and \eqref{eq:g-def},
\begin{equation}\label{eq:-g/e}
    -\frac{1}{\varepsilon}g(\varphi_0,\varepsilon)=\frac{R(v_-+\varepsilon\varphi_0)^5\varphi_0}{\kappa\alpha(\gamma-1)}\sigma(\varepsilon)\frac{\gamma p_-(\gamma+1)}{2\bigl(v_-+\tfrac{\gamma+1}{2}\varepsilon\bigr)}.
\end{equation}
Since $\hat K$ is compact and $v_->0$ on $\hat K$, there exists $\varepsilon_2>0$ such that
\[
    \varepsilon_2<\frac{1}{\gamma+1}\inf_{\hat K}v_-
    \qquad\text{and}\qquad
    v_-+\varepsilon\varphi_0>0\ \ \text{for }(\varphi_0,q)\in\hat K,\ |\varepsilon|<\varepsilon_2 .
\]
By \eqref{eq:-g/e}, $f$ extends to a $C^\infty$ function of $(\varphi_0,\varphi_1,\varphi_2,q,\varepsilon)$ on $U\times I$, where $I=(-\varepsilon_2,\varepsilon_2)$ and $U\subset\mathbb R^5$ is an open neighborhood of $\{(\varphi_0,\varphi_1,\varphi_2,q):(\varphi_0,q)\in\hat K\}$. This is \textnormal{(H1)}. In particular, at $\varepsilon=0$, \eqref{eq:-g/e} gives
\begin{equation} \label{eq:g-limit}
    \left.\left(-\frac{1}{\varepsilon}g(\varphi_0,\varepsilon)\right)\right|_{\varepsilon=0} = \frac{\sigma_-^3(\gamma+1)R v_-^5}{2\kappa\alpha(\gamma-1)}\varphi_0, \qquad \sigma_-=\sigma(0).
\end{equation}

The equilibria of the slow-fast system \eqref{system-slow} for $\varepsilon>0$ satisfy $\varphi_1=\varphi_2=0$ and $f(\varphi_0,0,0,\varepsilon)=0$. By \eqref{eq:f-def}, $f(\varphi_0,0,0,\varepsilon)=0$ is equivalent to
\[
  \sigma(\varepsilon)\varphi_0\left[\varphi_0\left(\frac{\sigma(\varepsilon)^2}{\gamma-1} + \frac{\sigma(\varepsilon)^2}{2}\right) + \frac{\sigma(\varepsilon)^2 v_- - \gamma p_-}{\varepsilon(\gamma-1)}\right]= 0,
\]
which by \eqref{eq:algebraic-cancellation} becomes
\[
    \sigma(\varepsilon)^3\Bigl(\frac{1}{\gamma-1}+\frac{1}{2}\Bigr)\varphi_0(\varphi_0-1)=0.
\]
Hence the equilibria with $\varphi_0\in K$ are $(0,0,0)$ and $(1,0,0)$ for every $\varepsilon\in(0,\varepsilon_2)$. Both are independent of $\varepsilon$.

\subsubsection{Verification of \textnormal{(H3)}}

At $\varepsilon=0$, the enlarged system \eqref{system-enlarged} gives $\varphi_2=0$ and $f(\varphi_0,\varphi_1,0,0)=0$. By \eqref{eq:f-def} and \eqref{eq:g-limit}, $f(\varphi_0,\varphi_1,0,0)=0$ reads
\begin{align*}
    0 &= \frac{R v_-^5}{\kappa\alpha} \left[\frac{\sigma_-^3(\gamma+1)}{2(\gamma-1)}(\varphi_0 - \varphi_0^2) - \frac{\mu\sigma_-^2}{\gamma-1}\varphi_1\right]+ \frac{(p_- - \sigma_-^2 v_-)v_-^4}{\kappa}\varphi_1.
\end{align*}
Hence the critical manifold is
\[
    \hat M_0:=\bigl\{(\varphi_0,\varphi_1,\varphi_2,q):(\varphi_1,\varphi_2)=h^0(\varphi_0,q),\ (\varphi_0,q)\in\hat K\bigr\},
\]
where
\begin{equation}\label{eq:h0-def}
    h^0(\varphi_0,q):=\bigl(C_*(q)(\varphi_0-\varphi_0^2),\,0\bigr)
\end{equation}
and, as in Theorem~\ref{thm:main},
\begin{equation}\label{eq:C1-def}
    C_* = \frac{R\gamma(\gamma+1)\sigma_-}{2(\mu R\gamma + \alpha(\gamma-1)^2)}.
\end{equation}
This is \textnormal{(H3)} with $Y=\hat K$ and $M_0=\hat M_0$. In particular, $\hat M_0$ is compact.

\subsubsection{Verification of \textnormal{(H2)}}
Since \textnormal{(H3)} holds, normal hyperbolicity in \textnormal{(H2)} means that the Jacobian $J$ of $(\varphi_2,f)$ with respect to $(\varphi_1,\varphi_2)$ at $\varepsilon=0$ has no eigenvalue on the imaginary axis on $\hat M_0$ (see Section~\ref{subsec:general-framework}). By \eqref{eq:f-def} and \eqref{eq:g-limit},
\begin{equation*} 
    J = \begin{pmatrix}
        0 & 1 \\
        -\dfrac{(\mu R\gamma + \alpha(\gamma-1)^2)p_- v_-^4}{\kappa \alpha (\gamma-1)} & -\dfrac{(R\kappa + \mu \alpha (\gamma-1) v_-^3) \sigma_- v_-}{\kappa \alpha (\gamma-1)}
    \end{pmatrix},
\end{equation*}
which is independent of $\varphi_0$, and
\begin{equation*} 
    \operatorname{tr}J = -\frac{(R\kappa+\mu\alpha(\gamma-1)v_-^3)\sigma_- v_-}{\kappa\alpha(\gamma-1)} < 0, \qquad
    \det J = \frac{(\mu R\gamma+\alpha(\gamma-1)^2)p_- v_-^4}{\kappa\alpha(\gamma-1)} > 0.
\end{equation*}
Hence both eigenvalues of $J$ have negative real parts, which is \textnormal{(H2)}.

\subsection{Slow Manifold and the Profile} \label{sec:construction}

\subsubsection{Slow Manifold and Endpoint Equilibria}\label{subsubsec:slow}

We reduce the slow-fast system \eqref{system-slow} on the slow manifold to the scalar equation \eqref{ODE2}. Let $\varepsilon_1$ and $h^\varepsilon$ be given by Proposition~\ref{prop:fenichel} for the enlarged system \eqref{system-enlarged}, and set $\varepsilon_3:=\tfrac12\min\{\varepsilon_1,\varepsilon_2\}$ with $\varepsilon_2$ from Section~\ref{subsubsec:H1}. By \eqref{eq:fenichel-expansion} and \eqref{eq:h0-def}, for $\varepsilon\in[0,\varepsilon_3]$,
\[
    h^\varepsilon=\bigl(C_*(\varphi_0-\varphi_0^2),\,0\bigr)+\varepsilon(h_1,h_2)\quad\text{on }\hat K,
    \qquad (h_1,h_2):=s.
\]
The functions $h_1$ and $h_2$ are smooth on $\hat K\times[0,\varepsilon_3]$, and
\begin{equation}\label{eq:h1-bound}
    |h_1|\le C_4 \quad\text{on } \hat K\times[0,\varepsilon_3].
\end{equation}

For $\varepsilon\in(0,\varepsilon_3]$, let $\hat M_\varepsilon$ be the three-dimensional manifold given by Proposition~\ref{prop:fenichel} for the enlarged system \eqref{system-enlarged}. It is locally invariant under \eqref{system-enlarged} and contains
\[
    \bigl\{(\varphi_0,\varphi_1,\varphi_2,q):(\varphi_1,\varphi_2)=h^\varepsilon(\varphi_0,q),\ (\varphi_0,q)\in\hat K\bigr\}.
\]
For $q\in\Omega$, set
\[
    M_\varepsilon:=\bigl\{(\varphi_0,\varphi_1,\varphi_2):(\varphi_0,\varphi_1,\varphi_2,q)\in\hat M_\varepsilon\bigr\}.
\]
We call $M_\varepsilon$ the slow manifold. It is a one-dimensional manifold, locally invariant under the slow-fast system \eqref{system-slow} since $\dot q=0$. It contains the graph
\begin{equation}\label{eq:M-eps}
    \big\{(\varphi_0, \varphi_1, \varphi_2): \varphi_0 \in \operatorname{int}K,\ \varphi_1 = C_*(\varphi_0 - \varphi_0^2) + \varepsilon h_1(\varphi_0, \varepsilon),\ \varphi_2 = \varepsilon h_2(\varphi_0, \varepsilon)\big\},
\end{equation}
which is one-dimensional and hence open in $M_\varepsilon$. On this graph, the equation $\dot\varphi_0=\varphi_1$ reads
\begin{equation}\label{ODE2}
    \dot{\varphi}_0 = F_\varepsilon(\varphi_0):=C_*(\varphi_0 - \varphi_0^2) + \varepsilon h_1(\varphi_0, \varepsilon),
    \qquad \varphi_0\in K.
\end{equation}
A prime on $F_\varepsilon$ denotes $d/d\varphi_0$. Since $h_1$ is smooth on $\hat K\times[0,\varepsilon_3]$, $\|F_\varepsilon\|_{C^5(K)}\le C$.

\begin{lemma}\label{lem:invariance}
Let $\varepsilon\in(0,\varepsilon_3]$ and $\varphi_0\in\operatorname{int}K$. Then
\begin{align}
    h_2(\varphi_0,\varepsilon)&=F_\varepsilon'(\varphi_0)\,F_\varepsilon(\varphi_0),\label{eq:inv-1}\\
    \varepsilon^2\,\partial_{\varphi_0}h_2(\varphi_0,\varepsilon)\,F_\varepsilon(\varphi_0)&=f\bigl(\varphi_0,\,F_\varepsilon(\varphi_0),\,\varepsilon h_2(\varphi_0,\varepsilon),\,\varepsilon\bigr).\label{eq:inv-2}
\end{align}
\end{lemma}

\begin{proof}
Let $\psi=(\psi_0,\psi_1,\psi_2)$ be the solution of the slow-fast system \eqref{system-slow} with
\[
    \psi(0)=\bigl(\varphi_0,\,F_\varepsilon(\varphi_0),\,\varepsilon h_2(\varphi_0,\varepsilon)\bigr).
\]
Since $\varphi_0\in\operatorname{int}K$, $\psi(0)$ lies on the graph \eqref{eq:M-eps}, and hence on $M_\varepsilon$. Let $V\supset M_\varepsilon$ be an open set as in Definition~\ref{def:loc-inv}. By continuity, there is $\delta_1>0$ with $\psi(z)\in V$ for $z\in[0,\delta_1]$, so $\psi(z)\in M_\varepsilon$ for $z\in[0,\delta_1]$ by the choice of $V$. Since the graph \eqref{eq:M-eps} is open in $M_\varepsilon$, there is $\delta\in(0,\delta_1]$ such that $\psi(z)$ lies on this graph for $z\in[0,\delta]$, that is,
\[
    \psi_1(z)=F_\varepsilon(\psi_0(z)),
    \qquad
    \psi_2(z)=\varepsilon h_2(\psi_0(z),\varepsilon),
    \qquad z\in[0,\delta].
\]
We differentiate these identities at $z=0$ from the right. Since $\dot\psi_0(0)=\psi_1(0)=F_\varepsilon(\varphi_0)$, the equations for $\dot\varphi_1$ and $\dot\varphi_2$ in \eqref{system-slow} give
\[
    F_\varepsilon'(\varphi_0)F_\varepsilon(\varphi_0)=h_2(\varphi_0,\varepsilon),
    \qquad
    \varepsilon^2\partial_{\varphi_0}h_2(\varphi_0,\varepsilon)F_\varepsilon(\varphi_0)
    =f\bigl(\varphi_0,F_\varepsilon(\varphi_0),\varepsilon h_2(\varphi_0,\varepsilon),\varepsilon\bigr),
\]
which are \eqref{eq:inv-1} and \eqref{eq:inv-2}.
\end{proof}

Set $\eta:=\tfrac14$, so that $[-\eta,1+\eta]\subset\operatorname{int}K$, and let $\varepsilon_4\in(0,\varepsilon_3]$ satisfy $C_4\varepsilon_4<\eta(1-\eta)\inf_\Omega C_*$. Then the definition of $F_\varepsilon$ in \eqref{ODE2} and the bound \eqref{eq:h1-bound} give
\begin{equation}\label{eq:F-close}
    |F_\varepsilon(\varphi_0)-C_*\varphi_0(1-\varphi_0)|<C_*\eta(1-\eta),
    \qquad \varphi_0\in K,\ \varepsilon\in(0,\varepsilon_4].
\end{equation}

\begin{lemma}\label{lem:endpoint}
Let $\varepsilon\in(0,\varepsilon_4]$. Then
\begin{align}
    &F_\varepsilon(0)=F_\varepsilon(1)=0,\label{eq:F-zeros}\\
    &h_2(0,\varepsilon)=h_2(1,\varepsilon)=0,\label{eq:h2-zeros}\\
    &F_\varepsilon(\varphi_0)>0 \qquad \text{for all } \varphi_0\in(0,1).\label{eq:F-positive}
\end{align}
In particular, $(0,0,0)$ and $(1,0,0)$ are equilibria of the slow-fast system \eqref{system-slow} lying on $M_\varepsilon$.
\end{lemma}

\begin{proof}
\textit{Step 1: $\{\varphi_0\in\operatorname{int}K:F_\varepsilon(\varphi_0)=0\}\subset\{0,1\}$.} If $F_\varepsilon(\varphi_0)=0$ for some $\varphi_0\in\operatorname{int}K$, then $h_2(\varphi_0,\varepsilon)=0$ by \eqref{eq:inv-1} and $f(\varphi_0,0,0,\varepsilon)=0$ by \eqref{eq:inv-2}. Thus $(\varphi_0,0,0)$ is an equilibrium of \eqref{system-slow}, and $\varphi_0\in\{0,1\}$ by Section~\ref{subsubsec:H1}.

\textit{Step 2: $F_\varepsilon(0)=F_\varepsilon(1)=0$.} Since $C_*\varphi_0(1-\varphi_0)\ge C_*\eta(1-\eta)$ on $[\eta,1-\eta]$, \eqref{eq:F-close} gives $F_\varepsilon>0$ on $[\eta,1-\eta]$, $F_\varepsilon(-\eta)<0$, and $F_\varepsilon(1+\eta)<0$. By the intermediate value theorem, $F_\varepsilon$ has a zero in $(-\eta,\eta)$ and a zero in $(1-\eta,1+\eta)$. Both intervals lie in $\operatorname{int}K$, and since $\eta<\frac12$, Step~1 identifies these zeros as $0$ and $1$. This proves \eqref{eq:F-zeros}, and \eqref{eq:inv-1} gives \eqref{eq:h2-zeros}.

\textit{Step 3: $F_\varepsilon>0$ on $(0,1)$.} If $F_\varepsilon(\varphi_0)\le0$ for some $\varphi_0\in(0,\eta)$, then, since $F_\varepsilon(\eta)>0$, the intermediate value theorem yields a zero of $F_\varepsilon$ in $(0,1)$, which contradicts Step~1. The same argument applies on $(1-\eta,1)$, and $F_\varepsilon>0$ on $[\eta,1-\eta]$ by Step~2. This proves \eqref{eq:F-positive}. Finally, \eqref{eq:F-zeros} and \eqref{eq:h2-zeros} show that $(0,0,0)$ and $(1,0,0)$ lie on the graph \eqref{eq:M-eps}, hence on $M_\varepsilon$.
\end{proof}

We prove Theorem~\ref{thm:main} with
\[
    \varepsilon_0:=\min\{\varepsilon_3,\varepsilon_4,\varepsilon_5,\varepsilon_6,\varepsilon_7\}.
\]
Later, $\varepsilon_5$ is fixed for the positivity of $\tilde\theta$ in Section~\ref{subsubsec:reconstruction}, $\varepsilon_6$ for the sign of $\tilde\theta_\xi$ in Section~\ref{subsubsec:pointwise}, and $\varepsilon_7$ for the exponential decay in Section~\ref{subsubsec:decay}. The constants $C_4,\dots,C_7$ depend only on $\Omega$ and $\gamma$, $R$, $\mu$, $\alpha$, $\kappa$, and so does $\varepsilon_0$.

\subsubsection{Construction of the Monotone Profile}\label{subsubsec:reconstruction}

For $\varepsilon\in(0,\varepsilon_0)$, we construct a solution $(\tilde v,\tilde u,\tilde\theta):\mathbb R\to\mathbb R_+\times\mathbb R\times\mathbb R_+$ of \eqref{eq:shock-ode}--\eqref{eq:shock-ode-bc} with $\tilde v_\xi>0$, $\tilde u_\xi<0$, and $\tilde v(0)=\frac{v_-+v_+}{2}$. By \eqref{eq:F-zeros} and \eqref{eq:F-positive}, the scalar initial value problem
\[
    \dot\varphi_0=F_\varepsilon(\varphi_0),\qquad \varphi_0(0)=\tfrac12,
\]
has a unique global solution $\varphi_0:\mathbb R\to(0,1)$, which is strictly increasing and satisfies
\begin{equation}\label{standingend}
    \lim_{z\to-\infty}\varphi_0(z)=0,\qquad \lim_{z\to+\infty}\varphi_0(z)=1.
\end{equation}
Set
\[
    \varphi(z):=\bigl(\varphi_0(z),\,F_\varepsilon(\varphi_0(z)),\,\varepsilon h_2(\varphi_0(z),\varepsilon)\bigr).
\]
Since $\varphi_0(z)\in(0,1)\subset\operatorname{int}K$, the invariance identities \eqref{eq:inv-1} and \eqref{eq:inv-2} hold at $\varphi_0=\varphi_0(z)$, and
\[
    \dot\varphi_0=\varphi_1,\qquad
    \varepsilon\dot\varphi_1=\varepsilon F_\varepsilon'(\varphi_0)F_\varepsilon(\varphi_0)=\varphi_2,\qquad
    \varepsilon\dot\varphi_2=\varepsilon^2\partial_{\varphi_0}h_2(\varphi_0,\varepsilon)F_\varepsilon(\varphi_0)=f(\varphi_0,\varphi_1,\varphi_2,\varepsilon).
\]
Hence $\varphi$ solves the slow-fast system \eqref{system-slow} on $\mathbb R$, and it connects $(0,0,0)$ to $(1,0,0)$ by the limits \eqref{standingend} and Lemma~\ref{lem:endpoint}.

Set
\[
    \tilde v(\xi):=v_-+\varepsilon\varphi_0(\varepsilon\xi),
\]
and define
\begin{equation}\label{eq:u,th-construction}
\begin{aligned}
    \tilde u&:=u_--\sigma(\tilde v-v_-),\\
    \tilde\theta&:=\frac{\tilde v}{R}\left[p_--\sigma^2(\tilde v-v_-)-\mu\sigma\frac{\tilde v'}{\tilde v}
        +\kappa\left(-\frac{\tilde v''}{\tilde v^5}+\frac{5(\tilde v')^2}{2\tilde v^6}\right)\right].
\end{aligned}
\end{equation}
Since $\tilde v_\xi=\varepsilon^2F_\varepsilon(\varphi_0)$ and $\tilde v_{\xi\xi}=\varepsilon^3F_\varepsilon'(\varphi_0)F_\varepsilon(\varphi_0)$, \eqref{eq:F-positive} and $\|F_\varepsilon\|_{C^5(K)}\le C$ give
\begin{equation}\label{eq:construction-v-bounds}
    0<\tilde v_\xi\le C\varepsilon^2,\qquad |\tilde v_{\xi\xi}|\le C\varepsilon|\tilde v_\xi|.
\end{equation}
Hence $\tilde u_\xi=-\sigma\tilde v_\xi<0$. Since $\varphi_0$ takes values in $(0,1)$ and $\varphi_0(0)=\frac12$, we also have $v_-<\tilde v<v_+$ and $\tilde v(0)=\frac{v_-+v_+}{2}$. For $\varepsilon\in(0,\varepsilon_4]$, the formula for $\tilde\theta$ in \eqref{eq:u,th-construction} and the bounds \eqref{eq:construction-v-bounds} give
\[
    |\tilde\theta-\theta_-|\le C\bigl(|\tilde v-v_-|+|\tilde v_\xi|+|\tilde v_{\xi\xi}|+|\tilde v_\xi|^2\bigr)\le C_5\varepsilon.
\]
Let $\varepsilon_5\in(0,\varepsilon_4]$ satisfy $C_5\varepsilon_5\le\frac12\inf_\Omega\theta_-$. Then $\tilde\theta\ge\theta_-/2>0$ for $\varepsilon\in(0,\varepsilon_5]$, so $(\tilde v,\tilde u,\tilde\theta)$ takes values in $\mathbb R_+\times\mathbb R\times\mathbb R_+$.

By construction, $\tilde u$ and $\tilde\theta$ satisfy the integrated mass and momentum equations \eqref{eq:int1} and \eqref{eq:int2}, under which the integrated energy equation \eqref{eq:int3} is equivalent to the third-order equation \eqref{eq:ODE-v}. Since $\varphi$ solves the slow-fast system \eqref{system-slow}, \eqref{eq:ODE-v} holds for $\tilde v$. Hence $(\tilde v,\tilde u,\tilde\theta)$ satisfies the integrated system \eqref{eq:int-system}, and therefore the profile equations \eqref{eq:shock-ode}. Since $\varphi$ connects $(0,0,0)$ to $(1,0,0)$, we have $\tilde v\to v_\pm$, $\tilde v'\to0$, and $\tilde v''\to0$ as $\xi\to\pm\infty$, and the Rankine--Hugoniot condition \eqref{eq:RH} gives $(\tilde u,\tilde\theta)\to(u_\pm,\theta_\pm)$. Hence $(\tilde v,\tilde u,\tilde\theta)$ solves \eqref{eq:shock-ode}--\eqref{eq:shock-ode-bc}. The sign of $\tilde\theta_\xi$ is proved in Section~\ref{subsubsec:pointwise}.

\section{Uniqueness and Pointwise Estimates} \label{sec:uniqueness_estimates}

\setcounter{equation}{0}

\subsection{Proof of Uniqueness} \label{subsec:uniqueness}

We show that every monotone solution of \eqref{eq:shock-ode}--\eqref{eq:shock-ode-bc} is a translate of the profile constructed in Section~\ref{subsubsec:reconstruction}. As in \cite{BDD06,HKKL25,HK26,HKO26}, we introduce the auxiliary variable
\begin{equation*} 
    \tilde{w}(\xi) := -\sqrt{\kappa} \frac{\tilde{v}'(\xi)}{\tilde{v}(\xi)^{5/2}},
\end{equation*}
so that the total energy \eqref{eq:total_energy} takes the form $\tilde{E} = R\tilde{\theta}/(\gamma-1) + \tilde{w}^2/2 + \tilde{u}^2/2$. By the integrated system \eqref{eq:int-system}, the profile $(\tilde{v}, \tilde{w}, \tilde{\theta})$ satisfies the first-order system
\begin{equation} \label{ODE:vwth}
\begin{aligned}
    \tilde{v}' &= -\frac{\tilde{v}^{5/2}}{\sqrt{\kappa}}\tilde{w}, \\
    \tilde{w}' &= \frac{\tilde{v}^{5/2}}{\sqrt{\kappa}}\left( \left[\sigma^2 (\tilde{v} - v_-) + (\tilde{p} - p_-)\right] - \frac{\mu \sigma}{\sqrt{\kappa}}\tilde{v}^{3/2}\tilde{w} \right),\\
    \tilde{\theta}' &= -\frac{\sigma \tilde{v}}{\alpha} \left( \frac{R}{\gamma-1}(\tilde{\theta}-\theta_-) - \frac{\sigma^2}{2}(\tilde{v}-v_-)^2 - \frac{\tilde{w}^2}{2} + p_-(\tilde{v}-v_-) \right),
\end{aligned}
\end{equation}
subject to 
\begin{equation*} 
    \lim_{\xi \to \pm\infty} (\tilde{v}, \tilde{w}, \tilde{\theta})(\xi) = (v_\pm, 0, \theta_\pm).
\end{equation*}
The left end state $(v_-,0,\theta_-)$ is an equilibrium of \eqref{ODE:vwth}.

The Jacobian matrix of the first-order system \eqref{ODE:vwth} at the left end state $(v_-,0,\theta_-)$ is
\begin{equation}\label{eq:B-matrix}
B=
\begin{pmatrix}
    0 & -b_1 & 0\\
    b_2 & -b_3 & b_4\\
    -b_5 & 0 & -b_6
\end{pmatrix},
\end{equation}
where
\begin{equation*}
\begin{aligned}
b_1 &:= \frac{v_-^{5/2}}{\sqrt{\kappa}}, & b_2 &:= \frac{v_-^{5/2}}{\sqrt{\kappa}}\left(\sigma^2 - \frac{p_-}{v_-}\right), & b_3 &:= \frac{\sigma \mu v_-^4}{\kappa}, \\
b_4 &:= \frac{R v_-^{3/2}}{\sqrt{\kappa}}, & b_5 &:= \frac{\sigma v_- p_-}{\alpha}, & b_6 &:= \frac{\sigma R v_-}{\alpha(\gamma-1)}.
\end{aligned}
\end{equation*}
By the identity \eqref{eq:algebraic-cancellation},
\begin{equation*}
    \operatorname{tr}B = -b_3 - b_6 < 0, \qquad
    \det B = \frac{R \sigma v_-^5}{\alpha \kappa(\gamma-1)}\left( \gamma p_- - \sigma^2 v_- \right)
    = \frac{R \sigma^3 v_-^5(\gamma+1)}{2\alpha\kappa(\gamma-1)}\,\varepsilon > 0.
\end{equation*}

Since the eigenvalues of $B$ have a negative sum and a positive product, $B$ has one positive eigenvalue $\lambda_u$, and the other two eigenvalues have negative real parts. Let $X=(X_1,X_2,X_3)\ne0$ satisfy $BX=\lambda_uX$. By \eqref{eq:B-matrix},
\[
    -b_1X_2=\lambda_uX_1,\qquad
    b_2X_1-b_3X_2+b_4X_3=\lambda_uX_2,\qquad
    -b_5X_1-b_6X_3=\lambda_uX_3.
\]
If $X_1=0$, then $b_1X_2=0$ and $(\lambda_u+b_6)X_3=0$, so $X_2=X_3=0$ since $b_1,b_6,\lambda_u>0$. Hence $X_1\ne0$. By the unstable manifold theorem \cite[Theorems~9.4 and 9.5]{T12}, $(v_-,0,\theta_-)$ has a one-dimensional local unstable manifold $W^u_{\mathrm{loc}}$ tangent to $X$. Since $X_1\ne0$, $W^u_{\mathrm{loc}}\setminus\{(v_-,0,\theta_-)\}$ consists of two orbits of \eqref{ODE:vwth}, one in $\{v>v_-\}$ and one in $\{v<v_-\}$.

Let $(\bar v,\bar u,\bar\theta)$ be a monotone solution of \eqref{eq:shock-ode}--\eqref{eq:shock-ode-bc}, and set $\bar w:=-\sqrt{\kappa}\,\bar v'/\bar v^{5/2}$. Integrating the profile equations \eqref{eq:shock-ode} once, we obtain constants $m_1$, $m_2$, and $m_3$ such that
\begin{equation}\label{eq:flux-constants}
\begin{aligned}
    &-\sigma(\bar v-v_-)-(\bar u-u_-)=m_1,\\
    &-\sigma(\bar u-u_-)+(\bar p-p_-)
     =\mu\frac{\bar u'}{\bar v}+\kappa\left(-\frac{\bar v''}{\bar v^5}+\frac{5(\bar v')^2}{2\bar v^6}\right)+m_2,\\
    &-\sigma(\bar E-E_-)+(\bar p\bar u-p_-u_-)
     =\alpha\frac{\bar\theta'}{\bar v}+\mu\frac{\bar u\bar u'}{\bar v}
     +\kappa\bar u\left(-\frac{\bar v''}{\bar v^5}+\frac{5(\bar v')^2}{2\bar v^6}\right)
     +\kappa\frac{\bar v'\bar u'}{\bar v^5}+m_3,
\end{aligned}
\end{equation}
where $\bar p:=p(\bar v,\bar\theta)$ and $\bar E:=\frac{R}{\gamma-1}\bar\theta+\frac{\kappa}{2}\frac{(\bar v')^2}{\bar v^5}+\frac{\bar u^2}{2}$. We show that $m_1=m_2=m_3=0$.

\textit{Step 1: $m_1=0$.} Letting $\xi\to-\infty$ in $\eqref{eq:flux-constants}_1$ and using the far-field conditions \eqref{eq:shock-ode-bc}, we obtain $m_1=0$. Hence $\bar u-u_-=-\sigma(\bar v-v_-)$ and $\bar u'=-\sigma\bar v'$.

\textit{Step 2: $m_2=0$.} By Step~1, $\eqref{eq:flux-constants}_2$ becomes
\begin{equation}\label{eq:w-linear}
    \bar w'+\Lambda\bar w=\Phi,
    \qquad
    \Lambda:=\frac{\mu\sigma}{\kappa}\bar v^4,
    \qquad
    \Phi:=\frac{\bar v^{5/2}}{\sqrt{\kappa}}\Bigl[\sigma^2(\bar v-v_-)+\bar p-p_--m_2\Bigr].
\end{equation}
By the far-field conditions \eqref{eq:shock-ode-bc}, $\Phi/\Lambda\to\bar w_-$ as $\xi\to-\infty$, where $\bar w_-:=-\sqrt{\kappa}\,m_2/(\mu\sigma v_-^{3/2})$. We show that $\bar w\to\bar w_-$ as $\xi\to-\infty$ and that $\bar w_-=0$. The derivative $\bar v'$ is integrable on $\mathbb R$, since $\bar v$ is nondecreasing and bounded. As $\bar v\ge v_-$, the function $\bar w$ is integrable as well, so $\bar w(\xi_n)\to0$ for some sequence $\xi_n\to-\infty$. In terms of $\bar w-\bar w_-$, the linear equation \eqref{eq:w-linear} reads
\[
    (\bar w-\bar w_-)'+\Lambda(\bar w-\bar w_-)=\Phi-\Lambda\bar w_- .
\]
Duhamel's formula and the bound $\Lambda\ge\mu\sigma v_-^4/\kappa$ give, for $\xi_n\le\xi$,
\[
    |\bar w(\xi)-\bar w_-|
    \le e^{-\mu\sigma v_-^4(\xi-\xi_n)/\kappa}|\bar w(\xi_n)-\bar w_-|
    +\frac{\kappa}{\mu\sigma v_-^4}\sup_{\tau\le\xi}\bigl|\Phi(\tau)-\Lambda(\tau)\bar w_-\bigr|.
\]
The term $\Phi-\Lambda\bar w_-$ tends to $0$ as $\xi\to-\infty$. Letting $n\to\infty$ and then $\xi\to-\infty$, we obtain $\bar w\to\bar w_-$. Since $\bar w(\xi_n)\to0$, we conclude that $\bar w_-=0$. Hence $m_2=0$ and $\bar w\to0$ as $\xi\to-\infty$.

\textit{Step 3: $m_3=0$.} By Steps~1 and~2, $\eqref{eq:flux-constants}_2$ and $\eqref{eq:flux-constants}_3$ give
\[
    -\sigma\left(\frac{R}{\gamma-1}(\bar\theta-\theta_-)-\frac{\kappa(\bar v')^2}{2\bar v^5}-\frac{\sigma^2}{2}(\bar v-v_-)^2\right)
    -\sigma p_-(\bar v-v_-)
    =\alpha\frac{\bar\theta'}{\bar v}+m_3 .
\]
By the far-field conditions \eqref{eq:shock-ode-bc} and $\kappa(\bar v')^2/\bar v^5=\bar w^2\to0$, the left-hand side tends to $0$ as $\xi\to-\infty$, so $\bar\theta'\to-v_-m_3/\alpha$. Since $\bar\theta$ converges, this limit is $0$, and hence $m_3=0$.

With $m_1=m_2=m_3=0$, the triple $(\bar v,\bar u,\bar\theta)$ satisfies the integrated system \eqref{eq:int-system}, and $(\bar v,\bar w,\bar\theta)$ solves the first-order system \eqref{ODE:vwth}. Since $\bar u$ is determined by $\bar v$ through Step~1, it suffices to show that
\begin{equation}\label{eq:translate}
    (\bar v,\bar w,\bar\theta)(\xi)=(\tilde v,\tilde w,\tilde\theta)(\xi+a),\qquad \xi\in\mathbb R,
\end{equation}
for some $a\in\mathbb R$, where $(\tilde v,\tilde u,\tilde\theta)$ is the solution constructed in Section~\ref{subsubsec:reconstruction}. Both triples tend to $(v_-,0,\theta_-)$ as $\xi\to-\infty$.

We first show that $\bar v>v_-$ on $\mathbb R$. If $\bar v(\xi_0)=v_-$ for some $\xi_0\in\mathbb R$, then, since $\bar v$ is nondecreasing and $\bar v\ge v_-$,
\[
    \bar v\equiv v_-,\qquad \bar w\equiv0\qquad\text{on }(-\infty,\xi_0],
\]
and $\eqref{ODE:vwth}_2$ gives $\bar p\equiv p_-$, that is, $\bar\theta\equiv\theta_-$, on $(-\infty,\xi_0]$. Thus $(\bar v,\bar w,\bar\theta)(\xi_0)=(v_-,0,\theta_-)$, so $(\bar v,\bar w,\bar\theta)\equiv(v_-,0,\theta_-)$ by the uniqueness of solutions to \eqref{ODE:vwth}, which contradicts $\bar v(\xi)\to v_+>v_-$ as $\xi\to+\infty$.

Since $(v_-,0,\theta_-)$ is hyperbolic and $(\bar v,\bar w,\bar\theta)\to(v_-,0,\theta_-)$ as $\xi\to-\infty$, there is $\xi_1\in\mathbb R$ such that $(\bar v,\bar w,\bar\theta)(\xi)\in W^u_{\mathrm{loc}}$ for $\xi\le\xi_1$. Since $\bar v>v_-$ and $\tilde v>v_-$, the points $(\bar v,\bar w,\bar\theta)(\xi)$ for $\xi\le\xi_1$ and $(\tilde v,\tilde w,\tilde\theta)(\xi)$ for $\xi$ near $-\infty$ lie on the orbit of $W^u_{\mathrm{loc}}$ in $\{v>v_-\}$. Hence $(\bar v,\bar w,\bar\theta)(\xi_1)=(\tilde v,\tilde w,\tilde\theta)(\xi_1+a)$ for some $a\in\mathbb R$, and \eqref{eq:translate} follows from the uniqueness of solutions to \eqref{ODE:vwth}.

Finally, Step~1, \eqref{eq:translate}, and the mass relation \eqref{eq:mass-relation} give
\[
    \bar u(\xi)-u_-=-\sigma\bigl(\bar v(\xi)-v_-\bigr)=-\sigma\bigl(\tilde v(\xi+a)-v_-\bigr)=\tilde u(\xi+a)-u_-,
    \qquad \xi\in\mathbb R.
\]
Since $\tilde v$ is strictly increasing, the normalization $\tilde v(0)=\frac{v_-+v_+}{2}$ in Theorem~\ref{thm:main} fixes $a$, so the monotone profile is unique.

\subsection{Proof of Pointwise Estimates} \label{sec:estimates}

\subsubsection{Derivative Estimates}\label{subsubsec:higher}\label{subsubsec:pointwise}
We first prove
\begin{equation}\label{eq:higher-target}
    |\partial_\xi^k(\tilde v,\tilde u,\tilde\theta)|\le C\varepsilon^{k-1}|\tilde v_\xi|,
    \qquad k=2,3,4.
\end{equation}
Define functions $P_k$ of $\varphi_0\in K$ by
\[
    P_1:=1,\qquad P_{k+1}:=F_\varepsilon'P_k+F_\varepsilon P_k',\qquad k=1,\dots,5,
\]
where a prime denotes $d/d\varphi_0$, so that $P_2=F_\varepsilon'$ and $P_3=(F_\varepsilon')^2+F_\varepsilon F_\varepsilon''$. Since $\dot\varphi_0=F_\varepsilon(\varphi_0)$, induction on $k$ gives
\[
    \frac{d^k\varphi_0}{dz^k}=F_\varepsilon(\varphi_0)\,P_k(\varphi_0),\qquad k=1,\dots,6.
\]
Since $P_k$ is a polynomial in $F_\varepsilon,\dots,F_\varepsilon^{(k-1)}$ and $\|F_\varepsilon\|_{C^5(K)}\le C$, we have $|P_k|\le C$ on $K$ for $k=1,\dots,6$, and hence, with $z=\varepsilon\xi$,
\begin{equation} \label{eq:v_higher_bounds}
    |\partial_\xi^k\tilde v|
    =\varepsilon^{k+1}\Bigl|\frac{d^k\varphi_0}{dz^k}\Bigr|
    \le C\varepsilon^{k+1}|F_\varepsilon(\varphi_0)|
    =C\varepsilon^{k-1}|\tilde v_\xi|,
    \qquad k=1,\dots,6 .
\end{equation}
By the mass relation \eqref{eq:mass-relation}, $\partial_\xi^k\tilde u=-\sigma\,\partial_\xi^k\tilde v$ for every $k\ge1$, so \eqref{eq:higher-target} holds for $\tilde v$ and $\tilde u$.

We next prove the estimate \eqref{eq:sigma-estimate} of $\sigma$, the relations \eqref{eq:derivative_relations} and the comparability \eqref{eq:derivative-comparability} of the derivatives, and $\tilde\theta_\xi<0$. The estimate \eqref{eq:sigma-estimate} is the bound \eqref{eq:sigma-small-amplitude} in Lemma~\ref{lem:RH-parametrization}. For the velocity, the mass relation \eqref{eq:mass-relation} gives
\begin{equation} \label{eq:u-prime-est}
    |\tilde{u}' + \sigma_- \tilde{v}'| = |-\sigma \tilde{v}' + \sigma_- \tilde{v}'| = |\sigma_- - \sigma| |\tilde{v}'| \le C\varepsilon |\tilde{v}'|.
\end{equation}
For the temperature, differentiating the formula \eqref{eq:theta_expression} for $\tilde\theta$ yields
\begin{equation} \label{eq:theta-prime}
    \tilde{\theta}' = \frac{1}{R} \left( p_- - 2\sigma^2 \tilde{v} + \sigma^2 v_- \right)\tilde{v}'
    -\frac{\mu\sigma}{R}\tilde v''
    +\frac{\kappa}{R}\left(-\frac{\tilde v'''}{\tilde v^4}+\frac{9\tilde v'\tilde v''}{\tilde v^5}-\frac{25(\tilde v')^3}{2\tilde v^6}\right).
\end{equation}
By \eqref{eq:v_higher_bounds} and the bound $\tilde v'\le C\varepsilon^2$ in \eqref{eq:construction-v-bounds}, the terms with $\tilde v''$, $\tilde v'''$, or $(\tilde v')^3$ are bounded by $C\varepsilon|\tilde v'|$.
Since $\tilde{v} - v_- = O(\varepsilon)$ and $\sigma^2 = \frac{\gamma p_-}{v_-} + O(\varepsilon)$ by \eqref{eq:sigma-small-amplitude}, the leading coefficient in \eqref{eq:theta-prime} satisfies
\begin{equation*}
    p_- - 2\sigma^2 \tilde{v} + \sigma^2 v_- = p_- - \sigma^2 v_- + O(\varepsilon) = (1-\gamma)p_- + O(\varepsilon).
\end{equation*}
Hence, for $\varepsilon\in(0,\varepsilon_4]$, \eqref{eq:theta-prime} gives
\begin{equation} \label{eq:theta-prime-est}
    \left|\tilde{\theta}' + \frac{(\gamma-1)p_-}{R} \tilde{v}'\right| \le C_6\varepsilon |\tilde{v}'|.
\end{equation}
The bounds \eqref{eq:u-prime-est} and \eqref{eq:theta-prime-est} are the relations \eqref{eq:derivative_relations}. Let $\varepsilon_6\in(0,\varepsilon_4]$ satisfy $C_6\varepsilon_6\le\tfrac{\gamma-1}{2R}\inf_\Omega p_-$. Then, for $\varepsilon\in(0,\varepsilon_6]$, \eqref{eq:theta-prime-est} yields $\tilde\theta'\le-\frac{(\gamma-1)p_-}{2R}\tilde v'<0$ and $|\tilde\theta'|\sim|\tilde v'|$. Since $|\tilde u'|=\sigma|\tilde v'|$ by the mass relation \eqref{eq:mass-relation} and $\sqrt{2/3}\,\sigma_-<\sigma<\sigma_-$ by the formula \eqref{eq:sigma-explicit} for $\sigma$, we also have $|\tilde u'|\sim|\tilde v'|$, which proves \eqref{eq:derivative-comparability}.

Finally, differentiating the formula \eqref{eq:theta-prime} for $\tilde\theta'$ gives
\[
\begin{aligned}
    \tilde\theta''={}&\frac1R\bigl(p_--2\sigma^2\tilde v+\sigma^2v_-\bigr)\tilde v''-\frac{2\sigma^2}{R}(\tilde v')^2
    -\frac{\mu\sigma}{R}\tilde v'''\\
    &+\frac{\kappa}{R}\left(-\frac{\tilde v''''}{\tilde v^4}+\frac{13\tilde v'\tilde v'''}{\tilde v^5}
    +\frac{9(\tilde v'')^2}{\tilde v^5}-\frac{165(\tilde v')^2\tilde v''}{2\tilde v^6}+\frac{75(\tilde v')^4}{\tilde v^7}\right).
\end{aligned}
\]
By \eqref{eq:v_higher_bounds} and the bound $\tilde v'\le C\varepsilon^2$ in \eqref{eq:construction-v-bounds}, each term is bounded by $C\varepsilon|\tilde v'|$. The same computation for $\partial_\xi^3\tilde\theta$ and $\partial_\xi^4\tilde\theta$ gives $|\partial_\xi^k\tilde\theta|\le C\varepsilon^{k-1}|\tilde v'|$ for $k=3,4$. This completes the proof of \eqref{eq:higher-target}.

\subsubsection{Factorization and Decay}\label{subsubsec:decay}

We prove the two-sided bound \eqref{eq:logistic-signed} and the approximation \eqref{eq:logistic-precise} of $\tilde v_\xi$, and the exponential decay in \eqref{eq:shock_estimates}. By Lemma~\ref{lem:endpoint} and the definition of $F_\varepsilon$ in \eqref{ODE2}, $h_1(0,\varepsilon)=h_1(1,\varepsilon)=0$ for $\varepsilon\in(0,\varepsilon_4]$. Define
\[
    \tilde h_1(\varphi_0,\varepsilon):=\frac{h_1(\varphi_0,\varepsilon)}{\varphi_0(1-\varphi_0)},\qquad \varphi_0\in(0,1).
\]
Since $1-\varphi_0\ge\frac12$ for $\varphi_0\le\frac12$ and $\varphi_0\ge\frac12$ for $\varphi_0\ge\frac12$, the mean value theorem applied at $\varphi_0=0$ and at $\varphi_0=1$ gives
\begin{equation}\label{eq:h1-tilde-bound}
    |\tilde h_1|\le2\sup_{\hat K\times[0,\varepsilon_3]}|\partial_{\varphi_0}h_1|\le C_7
    \quad\text{on } (0,1)\times(0,\varepsilon_4].
\end{equation}
In terms of $\tilde v$, the scalar equation \eqref{ODE2} reads
\begin{equation}\label{eq:v_differential_eq}
    \tilde v'=\bigl[C_*+\varepsilon\tilde h_1\bigr](\tilde v-v_-)(v_+-\tilde v).
\end{equation}
Hence, by the bound \eqref{eq:h1-tilde-bound} on $\tilde h_1$,
\[
    \bigl|\tilde v'-C_*(v_+-\tilde v)(\tilde v-v_-)\bigr|
    =\varepsilon|\tilde h_1|\,(v_+-\tilde v)(\tilde v-v_-)
    \le C\varepsilon\,(v_+-\tilde v)(\tilde v-v_-),
\]
which is \eqref{eq:logistic-precise}. Let $\varepsilon_7>0$ satisfy $C_7\varepsilon_7\le\tfrac12\inf_\Omega C_*$, so that $C_*+\varepsilon\tilde h_1\ge C_*/2>0$ for every $\varepsilon\in(0,\varepsilon_0)$. Since $v_-<\tilde v<v_+$, the factorization \eqref{eq:v_differential_eq} implies
\begin{equation*}
    C^{-1}(v_+ - \tilde{v})(\tilde{v} - v_-) \le \tilde{v}'(\xi) \le C(v_+ - \tilde{v})(\tilde{v} - v_-),
\end{equation*}
which is \eqref{eq:logistic-signed}.

The normalization $\tilde v(0)=\frac{v_-+v_+}{2}$ gives $v_+-\tilde v\ge\frac{\varepsilon}{2}$ for $\xi\le0$ and $\tilde v-v_-\ge\frac{\varepsilon}{2}$ for $\xi\ge0$. Integrating the two-sided bound \eqref{eq:logistic-signed} on each half-line then yields
\begin{equation} \label{eq:vdecay_final}
    |\tilde{v}(\xi)-v_-| \le C\varepsilon e^{-c\varepsilon|\xi|} \quad (\xi \le 0), \qquad |\tilde{v}(\xi)-v_+| \le C\varepsilon e^{-c\varepsilon|\xi|} \quad (\xi \ge 0).
\end{equation}
Hence, by the factorization \eqref{eq:v_differential_eq}, the derivative satisfies
\begin{equation}\label{eq:vprime-decay}
    |\tilde{v}'(\xi)| \le C\varepsilon^2 e^{-c\varepsilon|\xi|}.
\end{equation}

Finally, $|\tilde u-u_\pm|=\sigma|\tilde v-v_\pm|$ by the mass relation \eqref{eq:mass-relation} and the Rankine--Hugoniot condition. The formula \eqref{eq:theta_expression} for $\tilde\theta$, the bounds \eqref{eq:construction-v-bounds}, and the factorization \eqref{eq:v_differential_eq} yield
\[
    |\tilde\theta-\theta_\pm|
    \le C\bigl(|\tilde v-v_\pm|+|\tilde v'|+|\tilde v''|+|\tilde v'|^2\bigr)
    \le C\bigl(|\tilde v-v_\pm|+|\tilde v'|\bigr)
    \le C|\tilde v-v_\pm| .
\]
Hence, by the decay \eqref{eq:vdecay_final} of $\tilde v$,
\begin{equation*}
    |(\tilde{v} - v_\pm, \tilde{u} - u_\pm, \tilde{\theta} - \theta_\pm)| \le C \varepsilon e^{-c\varepsilon |\xi|},
    \qquad \pm\xi\ge0,
\end{equation*}
and, by the comparability \eqref{eq:derivative-comparability} and the decay \eqref{eq:vprime-decay} of $\tilde v'$,
\begin{equation*}
    |(\tilde{v}_\xi, \tilde{u}_\xi, \tilde{\theta}_\xi)| \le C \varepsilon^2 e^{-c\varepsilon |\xi|}.
\end{equation*}
This proves the exponential decay in \eqref{eq:shock_estimates} for every $\varepsilon\in(0,\varepsilon_0)$, and the proof of Theorem~\ref{thm:main} is complete.

\subsubsection{Proof of Corollary~\ref{cor:second-order}}\label{subsubsec:corollary}

Integrating the momentum equation $\eqref{eq:shock-ode}_2$ from $\pm\infty$ to $\xi$ and using the mass relation \eqref{eq:mass-relation} and the Rankine--Hugoniot condition, we obtain
\begin{equation}\label{eq:pressure-two-sided}
    \tilde p-p_\pm=-\sigma^2(\tilde v-v_\pm)-\frac{\mu\sigma\tilde v'}{\tilde v}
    +\kappa\left(-\frac{\tilde v''}{\tilde v^5}+\frac{5(\tilde v')^2}{2\tilde v^6}\right).
\end{equation}
The two identities in \eqref{eq:pressure-two-sided} yield
\begin{equation}\label{eq:secant-difference}
    \frac{\tilde p-p_+}{\tilde v-v_+}-\frac{\tilde p-p_-}{\tilde v-v_-}
    =\left[\frac{\mu\sigma\tilde v'}{\tilde v}
        +\kappa\left(\frac{\tilde v''}{\tilde v^5}-\frac{5(\tilde v')^2}{2\tilde v^6}\right)\right]
    \frac{\varepsilon}{(v_+-\tilde v)(\tilde v-v_-)}.
\end{equation}
In \eqref{eq:secant-difference}, the factorization \eqref{eq:v_differential_eq} and the bounds \eqref{eq:construction-v-bounds} give
\[
    \frac{\varepsilon\,\tilde v'}{(v_+-\tilde v)(\tilde v-v_-)}=\varepsilon\bigl[C_*+\varepsilon\tilde h_1\bigr],
    \qquad
    |\tilde v''|+(\tilde v')^2\le C\varepsilon\,\tilde v'.
\]
With the bound \eqref{eq:sigma-small-amplitude} on $\sigma$, $|\tilde v-v_-|\le\varepsilon$, and the bound \eqref{eq:h1-tilde-bound} on $\tilde h_1$, we obtain
\[
    \frac{\tilde p-p_+}{\tilde v-v_+}-\frac{\tilde p-p_-}{\tilde v-v_-}
    =\Bigl[\frac{\mu\sigma}{\tilde v}+O(\varepsilon)\Bigr]\varepsilon\bigl[C_*+\varepsilon\tilde h_1\bigr]
    =\frac{\mu\sigma_-C_*}{v_-}\varepsilon+O(\varepsilon^2).
\]

Finally, the definition \eqref{eq:C1-def} of $C_*$, $\sigma_-^2=\gamma p_-/v_-$, and the definition of $A_*$ give
\[
    \frac{\mu\sigma_-C_*}{v_-}
    =\frac{\mu R\gamma(\gamma+1)\sigma_-^2}{2v_-\bigl(\mu R\gamma+\alpha(\gamma-1)^2\bigr)}
    =\sigma_-A_*\frac{\mu R\gamma}{\mu R\gamma+\alpha(\gamma-1)^2},
\]
which proves \eqref{eq:secant-estimate}.

\section*{Funding}
This work was partially supported by the National Research Foundation of Korea [grant numbers RS-2024-00361663, NRF-2019R1A5A1028324] and by the Miwon Du-Myeong Fellowship.

\section*{Acknowledgments}
The author thanks Professor Moon-Jin Kang and Professor Jeongho Kim for their valuable comments.

\section*{Declaration of competing interest}
The author declares that he has no known competing financial interests or personal relationships that could have appeared to influence the work reported in this paper.

\bibliography{reference}

\end{document}